\documentclass{amsart}
\usepackage{amscd,amssymb}

\input xy
\xyoption{all}

\newcommand{\w}{\omega}

\newcommand{\diam}{\mathrm{diam}\,}
\newcommand{\dist}{\mathrm{dist}}

\newcommand{\id}{\mathrm{id}}
\newcommand{\asdim}{\mathrm{asdim}}
\newcommand{\pr}{\mathrm{pr}}

\newcommand{\IN}{\mathbb{N}}

\newcommand{\IZ}{\mathbb{Z}}

\newcommand{\U}{\mathcal{U}}

\newcommand{\e}{\varepsilon}

\newcommand{\tg}{\mathrm{tg}}

\newcommand{\fact}{\phi}

\newtheorem{theorem}{Theorem}[section]
\newtheorem{proposition}[theorem]{Proposition}
\newtheorem{lemma}[theorem]{Lemma}
\newtheorem{corollary}[theorem]{Corollary}
\newtheorem{example}[theorem]{Example}
\newtheorem{claim}[theorem]{Claim}

\title{Coarse classification of abelian groups and amenable shift-homogeneous metric spaces}
\author{Taras Banakh}
\address{Jan Kochanowski University, Kielce, Poland, and Ivan Franko University of Lviv, Ukraine}
\email{t.o.banakh@gmail.com}

\author{Matija Cencelj}
\address{Faculty of Education and IMFM, University of Ljubljana, P.O.B. 2964, Ljubljana 1001,
Slovenia}
\email{matija.cencelj@guest.arnes.si}

\author[D. Repov\v s]{Du\v{s}an~Repov\v{s}}
\address{Faculty of Education, and Faculty of Mathematics and Physics,
University of Ljubljana, P.O.B. 2964, Ljubljana, 1001, Slovenia}
\email{dusan.repovs@guest.arnes.si}

\author{Ihor Zarichnyi}
\address{Pidstryhach Institute for Applied Problems of Mechanics and Mathematics, Lviv, Ukraine}
\email{ihor.zarichnyj@gmail.com}


\begin{document}
\begin{abstract}In this paper we classify countable locally finite-by-abelian groups up to coarse isomorphism. This classification is derived from a coarse classification of amenable shift-homogeneous metric spaces.
\end{abstract}

\maketitle

\section{Introduction}

This paper is devoted to the problem of classification of countable  groups and homogeneous metric spaces up to coarse isomorphism or coarse equivalence, which is one of fundamental problems of Geometric Group Theory \cite{Gr1}, \cite{Gr2}, \cite{FM}. The main results of the paper are Theorems~\ref{t1.2}--\ref{t1.5} formulated in the introduction which is divided into two subsections. In the first part we survey the obtained results on the coarse classification of groups. These results will be derived from more general classification results
for shift-homogeneous metric spaces, presented in the second subsection of the introduction. The main theorems presented in the introduction will be proved in Sections~\ref{s5}, \ref{s6}, \ref{s8}, and \ref{s10}.

\subsection{Coarse classification of groups}

A function $f:G\to H$ between two groups is called
\begin{itemize}
\item a {\em coarse map} if for each finite subset $E\subset G$ there is a finite subset $D\subset H$ such that $f(x\cdot E)\subset f(x)\cdot D$ for each point $x\in G$;
\item a {\em coarse isomorphism} if $f$ is bijective and the maps $f$ and $f^{-1}$ are coarse.
\item a {\em coarse equivalence} if $f$ is coarse and there are a coarse function $g:H\to G$ and finite sets $D\subset G$ and $E\subset H$ such that $g\circ f(x)\in xD$ and $f\circ g(y)\in yE$ for any points $x\in G$ and $y\in H$.
\end{itemize}

Two groups $G,H$ are {\em coarsely equivalent} (resp. {\em coarsely isomorphic}) if there is a coarse equivalence (coarse isomorphism) $f:G\to H$. Since each coarse isomorphism is a coarse equivalence, any two coarsely isomorphic groups are coarsely equivalent.

The problem of classification of groups up to coarse equivalence was considered in \cite{BHZ}. Some final results were obtained for classes of locally finite and abelian groups. Let us recall that a group $G$ is called {\em locally finite} if each finitely generated subgroup of $G$ is finite. Locally finite and abelian groups belong to the class of locally finite-by-abelian groups.

A group $G$ is called {\em finite-by-abelian} if $G$ contains a finite normal subgroup $N\subset G$ with abelian quotient group $G/N$. By \cite{Noumann}, a locally finite group $G$ is finite-by-abelian if and only if it is an {\em FC-group} in the sense that each element $x\in G$ has finite conjugacy class $x^G=\{gxg^{-1}:g\in G\}$. A group $G$ is {\em locally finite-by abelian} if each finitely generated subgroup of $G$ is finite-by-abelian.

By the free rank $r_0(G)$ of a group $G$ we understand the number
$$r_0(G)=\sup\{n\in\w:\IZ^n\mbox{ is isomorphic to a subgroup of $G$}\}.$$ By Proposition~\ref{p9.4}, the free rank $r_0(G)$ of a locally finite-by-abelian group $G$ is equal to the asymptotic dimension $\asdim(G)$ of $G$.

The following classification of locally finite-by-abelian groups was proved in \cite{BHZ}.

\begin{theorem}\label{t1.1} Two countable locally finite-by-abelian groups $G,H$ are coarsely equivalent if and only if the free ranks of these groups coincide and either both groups are finitely generated or both are infinitely generated.
\end{theorem}

The classification of locally finite-by-abelian groups up to coarse isomorphism is more delicate and depends on so-called factorizing function of a group. This function is defined as follows.

For a subgroup $H$ of a group $G$ its {\em index spectrum}  $\sigma_{G:H}$ is the set of all possible indexes of $H$ in subgroups $E\subset G$ that contain $H$. The {\em factorizing function} of a subgroup $H\subset G$ is the function $\fact_{G/H}:\Pi\to\w\cup\{\infty\}$
defined on the set $\Pi$ of prime numbers and assigning to each prime number $p\in\Pi$ the (finite or infinite) number
$$\fact_{G/H}(p)=\sup\big\{k\in\w:p^k\mbox{ divides some finite number $i\in\sigma_{G:H}$}\big\}.$$

We shall say that two functions $f,g:\Pi\to\w\cup\{\infty\}$ are {\em almost equal} and write $f=^*g$ if $$\sum_{p\in P}|f(p)-g(p)|<\infty.$$
Here we assume that $\infty-\infty=0$ and $|\infty-n|=|n-\infty|=\infty$ for $n\in\w$.

Following \cite{BHZ} we say that a subgroup $H$ has {\em locally finite index} in $G$ if $H$ has finite index in each subgroup $\langle H\cup F\rangle$ generated by the union $H\cup F$ of $H$ with a finite subset $F\subset G$.

One of the principal results of this paper is the following theorem classifying locally finite-by-abelian groups up to coarse isomorphism.

\begin{theorem}\label{t1.2} Let $G_1,G_2$ be countable locally finite-by-abelian groups. In each group $G_i$ fix a subset $S_i\subset G_i$ of the smallest possible cardinality generating a subgroup $H_i=\langle S_i\rangle$ of locally finite index in $G_i$.
The groups $G_1$ and $G_2$ are coarsely isomorphic if and only if one of the following conditions holds:
\begin{enumerate}
\item $r_0(G_1)=r_0(G_2)=\infty$;
\item $r_0(G_1)=r_0(G_2)=0$ and $\fact_{G_1{/}H_1}=\fact_{G_2{/}H_2}$;
\item $0<r_0(G_1)=r_0(G_2)<\infty$ and $\fact_{G_1{/}H_1}=^*\fact_{G_2{/}H_2}$.
\end{enumerate}
\end{theorem}

Treating countable groups as homogeneous metric spaces, we shall derive Theorem~\ref{t1.2} from a more general Theorem~\ref{t1.5} classifying  amenable shift-homogeneous metric spaces. So, in the next subsection we shall survey some principal results on the coarse classification of homogeneous metric spaces.

\subsection{Coarse classification of homogeneous metric spaces}

A metric space $(X,d_X)$ is called
\begin{itemize}
\item {\em homogeneous} if for any points $x,y\in X$ there is an isometric bijection $f:X\to X$ such that $f(x)=y$;
\item {\em proper} if each closed bounded subset of $X$ is compact;
\item {\em boundedly-finite} if each bounded subset of $X$ is finite.
\end{itemize}

It is clear that each boundedly-finite metric space is proper.
The converse is true for countable homogeneous spaces: the Baire Theorem guarantees that a countable proper space $X$ contains an isolated points; then by the homogeneity, each point of $X$ is isolated and hence each compact subset of $X$ is finite.

It is clear that for each left-invariant metric $d$ on a group $G$, the metric space $(G,d)$ is homogeneous: for any points $x,y\in G$ the left shift $f:G\to G$, $f:g\mapsto yx^{-1}g$, is a bijective isometry with $f(x)=y$.
\smallskip

A map $f:X\to Y$ between metric spaces is called
\begin{itemize}
\item {\em coarse} if for each finite $\delta$ the {\em oscillation} $$\w_f(\delta)=\sup\{d_Y(f(x),f(x')):x,x'\in X \mbox{ and }d_X(x,x')\le\delta\}$$ is finite;
\item {\em a coarse isomorphism} if $f$ is bijective and the maps $f$ and $f^{-1}$ are coarse;
\item {\em a coarse equivalence} if $f$ is coarse and there is a coarse map $g:Y\to X$ such that the numbers $d_X(\id_X,g\circ f)=\sup_{x\in X}d_X(x,g\circ f(x))$ and $d_Y(\id_Y,f\circ g)=\sup_{y\in Y}d_Y(y,f\circ g(y))$ are finite.
\end{itemize}

It can be shown that two countable groups are coarsely isomorphic (coarsely equivalent) if and only if they are coarsely isomorphic (coarsely equivalent) as metric spaces endowed with proper left-invariant metrics. This shows that the notion of coarse isomorphism (coarse equivalence) of groups actually has a metric nature.

By \cite{Smith}, each countable group $G$ carries a proper left-invariant metric and such metric is unique up to coarse isomorphism. The latter means that for any two proper left-invariant metrics $d,\rho$ on $G$ the identity map $\id:(G,d)\to(G,\rho)$ is a coarse isomorphism of metric spaces.

From now on, each countable group $G$ is endowed with a proper left-invariant metric and is considered as a homogeneous metric space.

Many notions defined for groups can be naturally generalized to homogeneous metric spaces. In particular, finitely generated subgroups
correspond to large scale connected subsets of metric spaces.

Let $(X,d_X)$ be a metric space and $\e\ge 0$ be a real number. A subset $C\subset X$ is called {\em $\e$-connected} if for any two points $x,y\in X$ there is a sequence of points $x=x_0,x_1,\dots,x_n=y$ such that $d_X(x_{i},x_{i-1})\le\e$ for all $i\le n$. Such sequence $x_0,\dots,x_n$ will be called an {\em $\e$-chain of length $n$} linking the points $x,y$.

For a point $x\in X$ its {\em $\e$-connected component} $X_\e(x)$ is the maximal $\e$-connected subset of $X$ containing $x$. It consists of all points $y\in X$ that can be linked with $x$ by an $\e$-chain.
It will be convenient to put also $X_\infty(x)=X$.

A metric space $X$ is called {\em large scale connected} if it is $\e$-connected for some $\e<\infty$. It follows that a group $G$ endowed with a proper left-invariant metric is large scale connected if and only if it is finitely generated.

By a {\em factorizing step} of a metric space $X$ we understand the (finite or infinite) number
$$s(X)=\inf\big(\{\infty\}\cup\big\{\e\in[0,\infty):\forall \delta\in[\e,\infty)\;\forall x\in X \mbox{ the set $X_\delta(x)/\e=\{X_\e(y):y\in X_\delta(x)\big\}$ is finite}\big\}\big).$$
For a homogeneous metric space $X$ by $X_{s(X)}$ we denote any $s(X)$-connected component of $X$ (by homogeneity, all $s(X)$-connected components of $X$ are pairwise isometric).
The space $X_{s(X)}$ will be called {\em the factorizing component} of $X$.

The following theorem proved in Section~\ref{s6} can be considered as a metric analogue of Theorem~\ref{t1.1}.

\begin{theorem}\label{t1.3} Any two shift-homogeneous boundedly-finite metric spaces $X,Y$ are coarsely equivalent if and only if their factorizing components $X_{s(X)}$ and $Y_{s(Y)}$ are coarsely equivalent and the spaces $X,Y$ either both are large scale connected or both are not large scale connected.
\end{theorem}

Now we discuss the notion of shift-homogeneity appearing in this theorem.

A metric space $(X,d_X)$ is called {\em shift-homogeneous} if for any points $x,y\in X$ there is a bijective isometry $f:X\to X$ such that $f(x)=y$ and for any large scale connected subset $C\subset X$ the distance $d_X(f|C,\id)=\sup_{x\in C}d_X(f(x),x)$ is finite. Such isometry $f$ will be called a {\em shift-isometry} of $X$.

It is clear that each shift-homogeneous metric space is homogeneous. The converse is true if each large scale connected subset of $X$ is bounded.
In Proposition~\ref{p9.5} we shall prove that each countable locally finite-by-abelian group $G$ endowed with a proper left-invariant metric is shift-homogeneous.

Theorem~\ref{t1.3} yields a classification of shift-homogeneous metric spaces up to the coarse equivalence. Their classification up to the coarse isomorphism is more delicate and uses the notion of a factorizing function.

The {\em factorizing function} $\fact_X:\Pi\to\w\cup\{\infty\}$ of a boundedly-finite homogeneous metric space $X$ assigns to each prime number $p\in\Pi$ the (finite or infinite) number $$\fact_X(p)=\sup\big\{k\in\w:p^k\mbox{ divides the cardinality $|\{X_{s(X)}(x):x\in X_\delta\}|$ for some finite $\delta$}\big\}.$$
For a function $\fact:\Pi\to\w\cup\{\infty\}$ let
$$\IZ_\fact=\bigoplus_{p\in\Pi}\IZ_p^{\fact(p)}$$be the direct sum of cyclic groups $\IZ_p=\IZ/p\IZ$. If $\fact(p)=\infty$ then by $\IZ_p^\infty$ we understand the direct sum of countably many cyclic groups of order $p$. The group $\IZ_\fact$ will be endowed with a proper left-invariant metric.
It is easy to show that the factorizing step $s(\IZ_{\fact})=0$ and the factorizing function of the space $\IZ_\fact$ is equal to $\fact$.

Given two metric spaces $(X,d_X)$ and $(Y,d_Y)$ we endow their product $X\times Y$ with the metric
$$d_{X\times Y}\big((x,y),(x',y')\big)=\max\{d_X(x,x'),d_Y(y,y')\}.$$

Theorem~\ref{t1.3} will be derived from the following theorem proved in Section~\ref{s5}.

\begin{theorem}\label{t1.4} Each boundedly-finite shift-homogeneous metric space $X$ is coarsely isomorphic to the product $X_{s(X)}\times \IZ_{\fact_X}$.
\end{theorem}

Now we consider the problem of classification of shift-homogeneous metric spaces up to coarse isomorphisms. We shall obtain such classification only for amenable shift-homogeneous metric spaces.

A metric space $X$ is called {\em amenable} if for each positive $\e<\infty$ and $c>1$ there is a finite subset $F\subset X$ whose $\e$-neighborhood $O_\e(F)=\{x\in X:\exists y\in F\;d_X(x,y)\le\e\}$ is finite and has cardinality $|O_\e(F)|\le c\cdot |F|$. Amenable metric spaces will be discussed in Section~\ref{s7}. In particular, we shall prove that a countable group $G$ is amenable (which means that $G$ carries a right-invariant finitely additive probability measure defined on the algebra of all subsets of $G$) if and only if $G$ is amenable as a metric space endowed with a proper left-invariant metric.

The following theorem can be considered as a metric counterpart of Theorem~\ref{t1.2} and actually will be used in its proof.
In this theorem, each natural number $n\in\IN$ is considered as the (homogeneous) metric space $n=\{0,\dots,n-1\}$ endowed with a 2-valued metric. The factorizing function $\fact_n:\Pi\to\w$ assigns to each prime number $p$ the largest number $k=\fact_n(p)$ such that $p^k$ divides $n$. So, $n=\prod_{p\in\Pi}p^{\fact_n(p)}$.

\begin{theorem}\label{t1.5} Two amenable shift-homogeneous metric spaces $X,Y$ are coarsely isomorphic if and only if there are natural numbers $n,m\in\IN$ such that
\begin{enumerate}
\item $X_{s(X)}\times n$ is coarsely isomorphic to $Y_{s(Y)}\times m$;
\item $\fact_n\le \fact_X$, $\fact_m\le \fact_Y$, and $\fact_X-\fact_n=\fact_Y-\fact_m$.
\end{enumerate}
\end{theorem}

\section{Coarse classification of homogeneous ultrametric spaces}\label{s2}

In the proof of Theorems~\ref{t1.2}--\ref{t1.5} we shall use the coarse classification of homogeneous boundedly-finite ultrametric spaces obtained in \cite{BZ}. We recall that a metric space $X$ is called an {\em ultrametric space} if its metric $d_X$ satisfies the strong triangle inequality $$d_X(x,y)\le\max\{d_X(x,z),d_X(z,y)\}$$for all points $x,y,z\in X$.

By \cite{BZ}, any two unbounded homogeneous proper ultrametric spaces $X,Y$ are coarsely equivalent. Moreover, if these spaces are uncountable, then they are coarsely isomorphic.

The classification of {\em countable} homogeneous proper ultrametric spaces up to coarse isomorphism is more complicated. First observe that the Baire Theorem implies that each countable homogeneous proper ultrametric  space is boundedly-finite.

Observe also that for a point $x$ of an ultrametric space $X$ and any positive $\e$ the $\e$-connected component $X_\e(x)$ coincides with the closed $\e$-ball $O_\e(x)=\{y\in X:d_X(x,y)\le \e\}$ centered at $x$.  This implies that the factorizing step $s(X)$ of a boundedly-finite ultrametric space $X$ is equal to zero, and the factorizing function $\fact_X:\Pi\to\w\cup\{\infty\}$ of $X$ assigns to each prime number $p\in\Pi$ the number
$$\fact_X(p)=\sup\{k\in\w:\mbox{$p^k$ divides $|O_\e(x)|$ for some $x\in X$ and some $\e<\infty$}\}.$$

Now Theorem 9 of \cite{BZ} implies the following classification result:

\begin{theorem}\label{t2.1} Two homogeneous countable proper ultrametric spaces $X,Y$ are coarsely isomorphic if and only if $\fact_X=\fact_Y$.
\end{theorem}

This theorem implies:

\begin{corollary}\label{c2.2} Each homogeneous countable proper ultrametric space $X$ is coarsely isomorphic to the abelian group $\IZ_{\fact_X}$ (endowed with a proper left-invariant metric).
\end{corollary}

We shall also need the following known  result \cite{BDHM}.

\begin{proposition}\label{p2.3} A countable group $G$ carries a proper left-invariant ultrametric if and only if $G$ is locally finite.
\end{proposition}

\section{Factorizing step of a metric space}\label{s3}

In this section we discuss the notion of the factorizing step $s(X)$ of a metric space $X$. We recall that $s(X)=\inf ss(X)$ where $ss(X)$ is the set of all $\e\in[0,\infty]$ such that for each finite $\delta\ge\e$ and $x\in X$ the set $X_\delta(x)/\e=\{X_\e(y):y\in X_\delta(x)\}$ is finite.

In general, the factorizing step $s(X)$ does not belong to the set $ss(X)$.

\begin{example}\label{e3.1} For the closed subspace
$$X=\big\{\big(x+2\pi n,(-1)^n\tg(x)\big):x\in (-\tfrac{\pi}2,\tfrac{\pi}2),\;n\in\IZ\big\}$$of the Euclidean plane we get $ss(X)=(\pi,\infty]$ and hence $s(X)=\pi\notin ss(X)$.
\end{example}

On the other hand, we have

\begin{proposition}\label{p3.2} For any boundedly-finite homogeneous metric space $X$ we get $s(X)\in ss(X)$.
\end{proposition}

\begin{proof} Since the metric space $X$ is boundedly-finite and homogeneous, the set $d_X(X\times X)=\{d(x,y):x,y\in X\}$ is closed and discrete in $[0,\infty)$. Because of that we can find $s\in ss(X)$ such that the half-interval $(s(X),s]$ does not intersect the set $d(X\times X)$. This implies that $X_{s(X)}(x)=X_s(x)$ for each $x\in X$. Then for each finite $\delta\ge s$ and $x\in X$ the set $\{X_{s(X)}(y):y\in X_\delta(x)\}=\{X_s(x):x\in X_\delta(y)\}$ is finite, witnessing that $s(X)\in ss(X)$.
\end{proof}

Given a homogeneous metric space $X$ and a number $\e\in ss(X)$, let  $X/\e=\{X_\e(x):x\in X\}$ be the space of $\e$-connected components of $X$ endowed with the ultrametric
$$d_{X/\e}\big(X_\e(x),X_\e(y)\big)=\begin{cases} 0&\mbox{if }X_\e(x)=X_\e(y),\\
\inf\{\delta\ge\e: x\in X_\delta(y)\}&\mbox{otherwise}.
\end{cases}$$
It is clear that this ultrametric is well-defined if $\e>0$.

If $\e=0$, then $0=\e\in ss(X)$ implies that all components $X_\delta(x)$, $0\le \delta<\infty$, $x\in X$, are finite. Consequently, $X$ is boundedly-finite and by the homogeneity of $X$, the set $d(X\times X)$ is closed and discrete in $[0,\infty)$. This yields a positive $\delta>0$ such that each closed $\delta$-ball $O_\delta(x)$, $x\in X$, is a singleton, which implies that $X_\delta(x)=\{x\}$ and $d_{X/0}(X_0(x),X_0(y))\ge\delta>0$ iff $X_0(x)\ne X_0(y)$.

It follows from $\e\in ss(X)$ that the ultrametric space $X/\e$ is boundedly-finite. Moreover, the homogeneity of the metric space $X$ implies the homogeneity of the ultrametric space $X/\e$. By Corollary~\ref{c2.2}, the ultrametric space $X/\e$ is coarsely isomorphic to $\IZ_{\fact_{X/\e}}$.

By $q_\e:X\to X/\e$, $q_\e:x\mapsto X_\e(x)$, we shall denote the map assigning to each point $x\in X$ its $\e$-connected component $X_\e(x)$.

The proof of the following simple proposition is left to the reader.

\begin{proposition}\label{p3.3} A metric space $X$ has finite factorizing step $s(X)$ if and only if $X$ is coarsely isomorphic to a metric space $Y$ with finite factorizing step $s(Y)$.
\end{proposition}

\section{Factorization Theorem for shift-homogeneous metric spaces}\label{s4}

The following factorization theorem is the crucial step in the proof of Theorem~\ref{t1.4}.

\begin{theorem}\label{t4.1} For any shift-homogeneous metric space $X$ and any $\e\in ss(X)$ there is a coarse isomorphism $f:X_\e\times X/\e\to X$ such that for each $\e$-connected component $X_\e(x)\in X/\e$ the restriction $f|X_\e\times\{X_\e(x)\}$ is an isometry from $X_\e\times \{X_\e(x)\}$ onto $X_\e(x)\subset X$.
\end{theorem}

\begin{proof} Choose an unbounded increasing sequence of real numbers $(\e_n)_{n\in\w}$ with $\e_0=\e$. Fix any point $\theta\in X$ and identify the factorizing components $X_{\e_n}$, $n\in\w$,
with the $\e_n$-connected components $X_{\e_n}(\theta)$, $n\in\w$, of the point $\theta$.

For every $n\in\w$ the set $X_{\e_{n+1}}\setminus X_{\e_n}$ can be written as the disjoint union  $$X_{\e_{n+1}}\setminus X_{\e_n}=\bigcup_{x\in F_{n+1}}X_{\e_n}(x)$$ of $\e_n$-connected components $X_{\e_n}(x)$ of points $x\in F_{n+1}$ of some finite subset $F_{n+1}\subset X_{\e_{n+1}}$.

Using the shift-homogeneity of $X$, for each point $x\in F_{n+1}$ fix a shift-isometry $s_x:X\to X$ of $X$ such that $s_x(\theta)=x$. Then the restriction $s_X|X_{\e_n}:X_{\e_n}\to X_{\e_n}(x)$ is an isometry between the $\e_n$-connected components of $\theta$ and $x$, respectively.
This isometry induces an isometry $\tilde s_x:X_{\e_n}/\e\to X_{\e_n}(x)/\e$ between the subsets $X_{\e_n}/\e=\{X_\e(y):y\in X_{\e_n}\}$ and  $X_{\e_n}(x)/\e=\{X_\e(y):y\in X_{\e_n}(x)\}$ of the ultrametric  space $X/\e$. The latter isometry induces an isometry
$$p_n:X_\e\times X_{\e_n}(x)/\e\to X_\e\times X_{\e_n}/\e,\;\;\;p_n:(y,z)\mapsto(y,\tilde s_x^{-1}(z)).$$
Now we are able to define a coarse isomorphism $f:X_\e\times X/\e\to X$. For this observe that $$X/\e=\{X_\e\}\cup\bigcup_{n\in\w}\bigcup_{x\in F_{n+1}}X_{\e_n}(x)/\e.$$

Let $f_0:X_\e\times \{X_{\e}\}\to X_{\e}=X_{\e_0}$ be the projection onto the first coordinate. For every $n\in\w$ define a function $f_{n+1}:X_\e\times X_{\e_{n+1}}/\e\to X_{\e_{n+1}}$ letting
$f_{n+1}|X_\e\times X_{\e_n}/\e=f_n$ and $$f_{n+1}|X_\e\times X_{\e_n}(x)/\e=s_x\circ f_n\circ p_n$$for every $x\in F_{n+1}$. Taking into account that $s_x$, $x\in F_{n+1}$, are shift-isometries, one can check that the functions $f_n$, $n\in\w$, are coarse isomorphisms.

The functions $f_n$, $n\in\w$, compose a coarse isomorphism $f:X_\e\times X/\e\to X$ such that $f|X_\e\times X_{\e_n}/\e=f_n$ for all $n\in\w$.

It follows from the construction of $f$ that for any $\e$-connected component $X_\e(x)\in X/\e$, the restriction $f|X_\e\times\{X_\e(x)\}$ is an isometry from $X_\e\times \{X_\e(x)\}$ onto $X_\e(x)\subset X$.
\end{proof}

\section{Proof of Theorem~\ref{t1.4}}\label{s5}

Let $X$ is a boundedly-finite shift-homogeneous metric space. By Proposition~\ref{p3.2}, $s(X)\in ss(X)$ and by Theorem~\ref{t4.1}, the space $X$ is coarsely isomorphic to the product $X_{s(X)}\times (X/s(X))$. By definition, the factorizing function $\fact_X$ of $X$ coincides with the factorizing function $\fact_{X/s(X)}$ of the ultrametric space $X$.

By Corollary~\ref{c2.2}, the ultrametric space $X/s(X)$ is coarsely isomorphic to the group $\IZ_{\fact_{X/s(X)}}=\IZ_{\fact_X}$ endowed with a proper left-invariant metric. So, we obtain the following chain of coarse isomorphisms:
$$X\cong X_{s(X)}\times (X/s(X))\cong X_{s(X)}\times \IZ_{\fact_{X/s(X)}}=X_{s(X)}\times \IZ_{\fact_X}.$$

\section{Classification of shift-homogeneous metric spaces up to coarse equivalence}\label{s6}

By \cite[Corollary 7]{BZ}, each unbounded proper homogeneous ultrametric space is coarsely equivalent to the Cantor macro-cube $$2^{<\IN}=\{(x_n)_{n\in\IN}\in \{0,1\}^\IN:\exists m\in\IN\;\forall n\ge m\;\;x_n=0\}$$endowed with the metric
$$d\big((x_n),(y_n)\big)=\max_{n\in\IN}\,2^n|x_n-y_n|.$$ This fact combined with
Theorem~\ref{t4.1} implies the following two classification results for shift-homogeneous metric spaces.

\begin{theorem}\label{t6.1} A shift-homogeneous metric space $X$ for every finite $\e\in ss(X)$ is coarsely equivalent to:
\begin{enumerate}
\item $X_\e$ iff $X$ is large scale connected;
\item $X_\e\times 2^{<\IN}$ iff $X$ is not large scale connected.
\end{enumerate}
\end{theorem}

The following theorem implies Theorem~\ref{t1.3} announced in the Introduction.

\begin{theorem}\label{t6.2} For two shift-homogeneous metric spaces $X,Y$ with finite factorizing steps $s(X),s(Y)$ the following conditions are equivalent:
\begin{enumerate}
\item $X$ and $Y$ are coarsely equivalent;
\item for some real numbers $\e\in ss(X)$ and $\delta\in ss(Y)$ the spaces $X_\e$, $Y_\delta$ are coarsely equivalent and the spaces $X,Y$ either both are large scale connected or both are not large scale connected;
\item for any real numbers $\e\in ss(X)$ and $\delta\in ss(Y)$ the spaces $X_\e$, $Y_\delta$ are coarsely equivalent and the spaces $X,Y$ either both are large scale connected or both are not large scale connected.
\end{enumerate}
\end{theorem}

\begin{proof} The implication (3)$\Rightarrow$(2) is trivial and (2)$\Rightarrow$(1) follows from
Theorem~\ref{t6.1}.

(1)$\Rightarrow$(3) Assume that the spaces $X$ and $Y$ are coarsely equivalent.
Then $X$ is large scale connected if and only if so is its coarse copy $Y$. It remains to show that for any real numbers $\e\in ss(X)$ and $\delta\in ss(Y)$, the factorizing components $X_\e$ and $Y_\delta$ are coarsely equivalent.

The coarse equivalence of the spaces $X,Y$ yields two coarse maps $f:X\to Y$ and $g:Y\to X$ such that
$$D=\max\{d_X(\id_X,g\circ f),d_Y(\id_Y,f\circ g)\}<\infty.$$
Let $\e'=\max\{\e,D,\w_g(\delta)\}$ and $\delta'=\max\{\delta,D,\w_f(\e)\}$.

Fix any point $x_0\in X$ and let $y_0=f(x_0)\in Y$. For each $s\in[0,\infty)$ identify the factorizing components $X_s$ and $Y_s$ with the $s$-connected components $X_s(x_0)$ and $Y_s(y_0)$ of the points $x_0$, $y_0$, respectively.

Since $\e\in ss(X)$, the set $\{X_\e(x):x\in X_{\e'}\}$ is finite. This fact combined with the homogeneity of $X$ implies that $\sup \{d(x',X_\e):x\in X_{\e'}\}<\infty$. It follows that the identity embedding $i_X:X_\e\to X_{\e'}$ is a coarse equivalence and there is a coarse map $r_X:X_{\e'}\to X_\e$ such that $r_X\circ i_X$ coincides with the identity map of $X_\e$.

By analogy we can show that the identity embedding $i_Y:Y_\delta\to Y_{\delta'}$ is a coarse equivalence and find a coarse map $r_Y:Y_{\delta'}(y_0)\to Y_\delta(y_0)$ such that $r_Y\circ i_Y$ coincides with the identity map of $Y_\delta$.

It follows that $f(X_\e)\subset Y_{\w_f(\e)}\subset Y_{\delta'}$.
Also $\e'\ge\max\{D,\w_g(\delta)\}$ and $d_X(x_0,g(y_0))=d_X(x_0,g\circ f(x_0))\le D$ imply $g(Y_\delta)=g(Y_\delta(y_0))\subset X_{\w_g(\delta)}(g(y_0))\subset X_{\e'}(x_0)=X_{\e'}$.

So, the compositions $r_Y\circ f:X_\e\to Y_\delta$ and $r_X\circ g:Y_\delta\to X_\e$ are well-defined coarse maps witnessing that the spaces $X_\e$ and $Y_\delta$ are coarsely equivalent.
\end{proof}

\section{Amenable metric spaces}\label{s7}

In this section we discuss amenable metric spaces, which appear in Theorem~\ref{t1.5}.

Let us recall that a metric space $X$ is {\em amenable} if for each $c>1$ and $\e<\infty$ there is a finite subset $F\subset X$ whose $\e$-neighborhood $O_\e(F)$ is finite and has cardinality $|O_\e(F)|\le c\cdot|F|$.

It follows that each amenable homogeneous metric space $X$ is boundedly-finite and hence proper.
The choice  of the term ``amenable'' is justified by the following proposition.

\begin{proposition}\label{p7.1} Each amenable metric space $X$ admits a finitely additive probability measure $\mu:\mathcal P(X)\to[0,1]$ defined on the algebra $\mathcal P(X)$ of all subsets of $X$, which is invariant in the sense that $\mu(f(A))=\mu(A)$ for any subset $A\subset X$ and any bijection $f:X\to X$ with $d_X(f,\id)<\infty$.
\end{proposition}

\begin{proof} By the definition of amenability, for each $n\in\IN$ there is  a finite subset $F_n\subset X$ such that $|O_n(F_n)|\le (1+\frac1n)|F_n|$ and hence $|O_n(F)\setminus F_n|\le\frac1n|F_n|$. Choose a free ultrafilter $\U$ on the set of positive integers $\IN$ and define a probability finitely additive measure $\mu$ on $X$ by the formula
$$\mu(A)=\lim_{n\to\U}\frac{|A\cap F_n|}{|F_n|}$$for $A\subset X$.

It remains to show that the measure $\mu$ is invariant. Given any bijection $f:X\to X$ of $X$ with $r=d_X(f,\id)<\infty$, it suffices to prove that $\mu(f(A))\le \mu(A)$. For every $n> r$ we get $f^{-1}(F_n)\subset O_r(F_n)\subset O_n(F_n)$ and hence
$$|f(A)\cap F_n|=|A\cap f^{-1}(F_n)|\le |A\cap O_n(F_n)|\le |A\cap F_n|+|O_n(F_n)\setminus F_n|\le|A\cap F_n|+\tfrac1n|F_n|.$$
Now we see that
$$\mu(f(A))=\lim_{n\to\U}\frac{|f(A)\cap F_n|}{|F_n|}=\lim_{n\to\U}\frac{|A\cap f^{-1}(F_n)|}{|F_n|}\le
\lim_{n\to\U}\frac{|A\cap F_n|}{|F_n|}+\lim_{n\to\U}\frac1n\frac{|F_n|}{|F_n|}=\mu(A)+0=\mu(A).$$
\end{proof}

The following corollary of Proposition~\ref{p7.1} implies that for countable groups our definition of amenability is equivalent to the classical one \cite{Amen}.

\begin{proposition}\label{p7.2} A countable group $G$ endowed with a proper left-invariant metric $d$ is amenable if and only if the group $G$ admits a right-invariant finitely additive probability measure $\mu:\mathcal P(G)\to[0,1]$ defined on the algebra of all subsets of $G$.
\end{proposition}

\begin{proof} If the metric space $(G,d)$ is amenable, then by Proposition~\ref{p7.1} there is a finitely-additive probability measure $\mu:\mathcal P(G)\to[0,1]$ which is invariant in the sense that $\mu(f(A))=\mu(A)$ for any subset $A\subset X$ and any bijective function $f:G\to G$ with $d(f,\id)<\infty$. We claim that the measure $\mu$ is right-invariant in the sense that $\mu(Ag)=\mu(A)$ for each subset $A\subset X$ and an element $g\in G$. The left-invariance of the metric $d$ implies that the right shift $r_g:G\to G$, $r_g:x\mapsto xg$, is a bijection with
$$d(r_g,\id)=\sup_{x\in G}d(r_g(x),x)=\sup_{x\in G}d(xg,x)=d(g,1_G)<\infty.$$ Then $\mu(Ag)=\mu(r_g(A))=\mu(A)$ by the invariantness of the measure $\mu$.
This proves the ``only if'' part of the proposition.

The ``if'' part follows from the classical F\o lner characterization of amenable groups \cite[4.10]{Amen}.
\end{proof}

The amenability of homogeneous metric spaces is inherited by their large scale connected components.

\begin{proposition}\label{p7.3} If a homogeneous metric space $X$ is amenable, then each $\e$-connected component of $X$ is amenable.
\end{proposition}

\begin{proof} To show that each $\e$-connected component of $X$ is amenable, fix any $c>1$ and $\delta<\infty$. By the amenability of the space $X$ there is a finite subset $F\subset X$ such that $|O_\delta(F)|<c\cdot|F|$. Choose a subset $E\subset F$ such that $F\subset \bigcup_{x\in E}X_\e(x)$ and $X_\e(x)\cap X_\e(y)=\emptyset$ for any distinct points $x,y\in E$. For each $x\in E$ let $F_x=F\cap X_\e(x)$. It follows from $\bigcup_{x\in E}O_\delta(F_x)\cap X_\e(x)\subset O_\delta(F)$ that
$$\sum_{x\in E}|O_\delta(F_x)\cap X_\e(x)|\le |O_\delta(F)|<c\cdot |F|=c\cdot \Big(\sum_{x\in E}|F_x|\Big)$$and hence
$|O_\delta(F_x)\cap X_\e(x)|<c\cdot |F_x|$ for some $x\in E$. Now we see that the $\e$-component $X_\e(x)$ of the point $x$ contains a finite subset $F_x$ witnessing that $X_\e(x)$ is amenable. Since $X$ is homogeneous all $\e$-connected components of $X$ are isometric to $X_\e(x)$ and also contain such a finite set.
\end{proof}

The following two lemmas will be used in the proof of Theorem~\ref{t1.5}.

\begin{lemma}\label{l7.4} Let $X$ be an amenable metric space. Two finite subsets $A,B$ of a metric space $Z$ have the same cardinality if there is a bijection $f:X\times Z\to X\times Z$ such that $f(X\times A)=X\times B$ and  $d_{X\times Z}(f|X\times A,\id)<\infty$.
\end{lemma}

\begin{proof} Assume conversely that $|A|>|B|$. Let $r=\max\{\diam(A\cup B),d_{X\times Z}(f|X\times A,\id)\}$. The amenability of $X$ yields a finite set $F\subset X$ such that
$|O_{2r}(F)|<\frac{|A|}{|B|}\cdot|F|$. It follows that $$f(F\times A)\subset O_r(F\times A)\subset O_r(O_r(F\times B))=O_{2r}(F\times B).$$ Since $f(F\times A)\subset f(X\times A)=X\times B$, we get
$f(F\times A)\subset O_{2r}(F\times B)\cap (X\times B)=O_{2r}(F)\times B$ and hence $|F|\cdot|A|=|F\times A|=|f(F\times A)|\le |O_{2r}(F)|\cdot |B|$.
Then $|O_{2r}(F)|\ge \frac{|A|}{|B|}\cdot|F|$, which contradicts the choice of $F$.

By analogy we can treat the case $|A|<|B|$.
\end{proof}

\begin{lemma}\label{l7.5} Let $Z$ be a boundedly-finite ultrametric space and $C$ be an $s$-connected amenable metric space for some positive real number $s$. If  $f:C\times Z\to X$ is a coarse isomorphism from the product $C\times Z$ onto a shift-homogeneous metric space $X$, then $\w_f(s)\in ss(X)$ and
for each $\e\ge \w_f(s)$ there are a number $n\in\IN$ and an $n$-to-1 function $f_\e:Z\to X/\e$ making the following diagram commutative:
$$
\begin{CD}
C\times Z@>f>> X\\
@V{pr}VV @VV{q_\e}V\\
Z@>>{f_\e}> X/\e
\end{CD}
$$
\end{lemma}

\begin{proof} To see that $\w_f(s)\in ss(X)$, it suffices to check that for any finite $\delta\ge \w_f(s)$ the set $\{X_{\w_f(s)}(y):y\in X_\delta(x)\}$ is finite for each $x\in X$.

Observe that for each $z\in Z$ the set $f(C\times\{z\})$ is $\w_f(s)$-connected and hence lies in some $\delta$-connected component of $X$. This fact can be used to show that $f^{-1}(X_\delta(x))=C\times Z_x$ for some set $Z_x\subset Z$. The $\delta$-connectedness of $X_\delta(x)$ implies the $\w_{f^{-1}}(\delta)$-connectedness of the set $f^{-1}(X_\delta(x))$ and its projection $Z_x$ onto $Z$. Since $Z$ is a boundedly-finite ultrametric space, the large scale connected set $Z_x\subset Z$ is bounded and hence finite. Since for each $y\in X_\delta(x)$ the preimage $f^{-1}(X_{\w_f(s)}(y))$ is of the form $C\times Z_y$ for some set $Z_y\subset Z_x$, we conclude that the set $\{X_{\w_f(s)}(y):y\in X_\delta(x)\}$ has finite cardinality $\le|Z_x|$, witnessing that $\w_f(s)\in ss(X)$.
\smallskip

Now fix any $\e\ge \w_f(s)$. We already know that for each $x\in X$ the preimage $f^{-1}(X_\e(x))$ is equal to the product $C\times Z_x$ for some finite set $Z_x\subset Z$. We claim that $|Z_x|=|Z_y|$ for all points $x,y\in X$. By the shift-homogeneity of $X$, there is a shift-isometry $h:X\to X$ such that $h(x)=y$. Then $h(X_\e(x))=X_\e(y)$ and $g=f^{-1}\circ h\circ f$ is a coarse isomorphism of $C\times Z$ with $g(C\times Z_x)=C\times Z_y$. We claim that $\dist(g|C\times Z_x,\id)<\infty$. Since $h$ is a shift-isometry, the number $r=\sup\{d_X(h(y),y):y\in X_\e(x)\}$ is finite.
Then for each $z\in Z_x$ and $c\in C$ we get
$$d_{C\times Z}(g(c,z),(c,z))=d_{C\times Z}(f^{-1}\circ h\circ f(c,z),f^{-1}\circ f(c,z))\le \w_{f^{-1}}(d_X(h\circ f(c,z),f(c,z))\le\w_{f^{-1}}(r).$$ Now it is legal to apply Lemma~\ref{l7.4} and conclude that $|Z_x|=|Z_y|=n$ for some number $n\in\IN$.

It follows that the map $f_\e:Z\to X/\e$ defined by $f_\e^{-1}(X_\e(x))=Z_x$ for $x\in X$ is $n$-to-1.
Taking into account that $f^{-1}(X_\e(x))=C\times Z_x=C\times f_\e^{-1}(X_\e(x))$ for any $x\in X$, we see that the diagram
 $$
\begin{CD}
C\times Z@>f>> X\\
@V{pr}VV @VV{q_\e}V\\
Z@>>{f_\e}> X/\e
\end{CD}
$$
commutes.
\end{proof}

\section{Proof of Theorem~\ref{t1.5}}\label{s8}

Let $X,Y$ be amenable shift-homogeneous metric spaces. Being amenable, the spaces $X,Y$ are boundedly-finite. Consequently, $s(X)\in ss(X)$ and $s(Y)\in ss(Y)$ by Proposition~\ref{p3.2} and hence the factorization functions $\fact_X=\fact_{X/s(X)}:\Pi\to\w\cup\{\infty\}$ and
$\fact_Y=\fact_{Y/s(Y)}:\Pi\to\w\cup\{\infty\}$ are well-defined.

To prove the ``only if'' part of Theorem~\ref{t1.5}, assume that there is a coarse isomorphism $i:X\to Y$. If $s(X)=\infty$, then $s(Y)=\infty$ by Proposition~\ref{p3.3}. In this case $X=X_{s(X)}$, $Y=Y_{s(Y)}$ and we can put $n=m=1$.

So, we assume that the proper steps $s(X)$ and $s(Y)$ are finite. By Theorem~\ref{t4.1}, there are coarse isomorphisms
$f:X_{s(X)}\times X/{s(X)}\to X$ and $g:Y_{s(Y)}\times Y/s(Y)\to Y$. Consider the coarse isomorphism $h=i\circ f:X_{s(X)}\times X/{s(X)}\to Y$ and put $\e=\max\{\w_h(s(X)),\w_g(s(Y))\}$.

By Lemma~\ref{l7.5}, there are functions $h_{\e}:X/{s(X)}\to Y/\e$ and $g_\e:Y/s(Y)\to Y/\e$ making the following diagram commutative:
$$\xymatrix{
 X_{s(X)}\times X/s(X)\ar[d]_{\pr}\ar[r]^-f&X\ar[ld]^{q_{s(X)}}\ar[r]^{i}&Y\ar[d]_{q_\e}&Y_{s(Y)}\times Y/s(Y)\ar[l]_-g\ar[d]^{\pr}\\
X/s(X)\ar[rr]_{h_\e}&&Y/\e&Y/s(Y)\ar[l]^{g_\e}.
}
$$
Moreover, the function $h_\e$ is $n$-to-1 and $g_\e$ is $m$-to 1 for some numbers $n,m\in\IN$.

\begin{claim}\label{cl:rev} $\fact_{X/s(X)}=\fact_n+\fact_{Y/\e}$.
\end{claim}

\begin{proof} The inequality $\fact_{X/s(X)}\le \fact_n+\fact_{Y/\e}$ will follow as soon as we check that $k\le \fact_n(p)+\fact_{Y/\e}(p)$ for any prime number $p$ and any non-negative integer number $k\le \fact_{X/s(X)}(p)$. By the homogeneity of $X$ and the definition of $\fact_{X/s(X)}(p)$, there is $\delta>s(X)$ such that $p^k$ divides $|\{X_{s(X)}(x'):x'\in X_\delta(x)\}|$ for every $x\in X$.

Recall that $i:X\to Y$ is a coarse isomorphism and consider the number $\delta'=\w_i(\delta)$. It follows that
$i^{-1}(Y_{\delta'})=\bigcup\{X_{\delta}(x):x\in i^{-1}(Y_{\delta'})\}$ and consequently
the cardinality of the set $\{X_{s(X)}(x):x\in i^{-1}(Y_{\delta'})\}$ is divided  by $p^k$.

Since $\w_i(s(X))\le \e$, for every $y\in Y$, the preimage $i^{-1}(Y_\e(y))$ is a union of $s(X)$-connected components of $X$ and the number of these components is equal to $n$ (as $h_\e$ is an $n$-to-1 map). Consequently, the cardinality of $|Y_{\delta'}/\e|$ is divided by $p^k/p^{\fact_n(p)}$ and hence $\fact_{Y/\e}(p)\ge k-\fact_n(p)$ and $k\le \fact_n(p)+\fact_{Y/\e}(p)$.
\smallskip

Next, we prove that $\fact_{X/s(X)}\ge \fact_n+\fact_{Y/\e}$. Given any prime number $p$ and any
non-negative integer number $k\le \fact_{Y/\e}(p)$, we need to find $\delta$ such that the cardinality of the set $\{X_{s(X)}(x):x\in X_\delta\}$ is divided by $p^{k+\fact_n(p)}$. Since $k\le \fact_{Y/\e}(p)$, there is $\delta'<\infty$ such that $p^k$ divides the cardinality of the set $|\{Y_\e(y):y\in Y_{\delta'}\}|$. Put $\delta=\w_{i^{-1}}(\delta')$ and observe that $i(X_\delta)$ is a finite union of $\delta'$-components, which implies that $|\{Y_\e(y):y\in i(X_\delta)\}|$ is divided by $p^k$. Since the map $h_\e:X/s(X)\to Y/\e$ is $n$-to-1, the cardinality of the set $\{X_{s(X)}(x):x\in X_\delta\}$ is divided by $p^k\cdot n$ and hence $k+\fact_n(p)\le \fact_{X/s(X)}(p)$.
\end{proof}

Claim~\ref{cl:rev} guarantees that
$\fact_{X/s(X)}=\fact_n+\fact_{Y/\e}$ and hence $\fact_n\le \fact_{X/s(X)}=\fact_X$.

By analogy the $m$-to-1 property of the function $g_\e:Y/s(Y)\to Y/\e$ can be used to prove that $\fact_{Y/s(Y)}=\fact_m+\fact_{Y/\e}$ and $\fact_m\le \fact_{Y/s(Y)}=\fact_Y$. Adding the function $\fact_m$ to the equation $\fact_{X/s(X)}=\fact_n+\fact_{Y/\e}$ and $\fact_n$ to the equation $\fact_{Y/s(Y)}=\fact_m+\fact_{Y/\e}$, we get
$$\fact_Y+\fact_n=\fact_{Y/s(Y)}+\fact_n=\fact_n+\fact_m+\fact_{Y/\e}=\fact_{X/s(X)}+\fact_m=\fact_X+\fact_m,$$
which implies $\fact_X-\fact_n=\fact_Y-\fact_m$.
\smallskip

To prove the ``if'' part, assume that there are numbers $n,m\in\IN$ satisfying the conditions (1)--(2) of the theorem. By Theorem~\ref{t4.1}, the spaces $X$ and $Y$ are coarsely isomorphic to the products $X_{s(X)}\times X/s(X)$ and $Y_{s(Y)}\times Y/s(Y)$, respectively.

By the condition (2) of Theorem~\ref{t1.5}, the function $\fact_Z=\fact_{X/s(X)}-\fact_n=\fact_{Y/s(Y)}-\fact_m$ is well-defined. Consider the locally finite abelian group $Z=\IZ_{\fact_Z}$ endowed with a proper left-invariant ultrametric. By Theorem~\ref{t2.1}, the product $Z\times n$ is coarsely isomorphic to $X/s(X)$ and the product $Z\times m$ is coarsely isomorphic to $Y/s(Y)$. Then we have the following chain of coarse isomorphisms
 $$X\cong X_{s(X)}\times X/s(X)\cong X_{s(X)}\times n\times Z\cong Y_{s(Y)}\times m\times Z\cong X_{s(Y)}\times Y/s(Y)\cong Y.$$

\section{Shift-homogeneous groups with finite factorizing step}\label{s9}

In this section we shall detect groups which are shift-homogeneous or have finite factorizing step.

Observe that each countable group $G$ (endowed with a proper left-invariant metric) is a boundedly-finite homogeneous metric space.  The $\e$-connected component $G_\e$ of the neutral element  $1_G\in G$ coincides with the subgroup generated by the closed $\e$-ball $O_\e(1_G)$ while the space $G/\e$ of $\e$-connected components coincides with the space $G/G_\e=\{xG_\e:x\in G\}$ of left cosets of the group $G$ by the subgroup $G_\e$. The space $G/G_\e$ carries the $G$-invariant ultrametric $d_{G/G_\e}(xG_\e,yG_\e)=\inf\{\delta\ge 0:xG_\delta=yG_\delta\}$.

Now the corresponding definitions imply the following characterization of countable groups with finite factorizing step.

\begin{proposition}\label{p9.1} A countable group $G$ endowed with a proper left-invariant metric has finite factorization step if and only if it contains a finitely generated subgroup of locally finite index in $G$.
\end{proposition}

Subgroups of locally finite index can be used to calculate the free rank of a countable locally finite-by-abelian group. The following lemma holds for a more general class of locally abelian-by-finite groups. We recall that a group $G$ is {\em abelian-by-finite} if it contains a normal abelian subgroup of finite index. A group $G$ is {\em locally abelian-by-finite} if each finitely generated subgroup of $G$ is abelian-by-finite. By \cite[9.6]{BHZ}, each finitely generated finite-by-abelian group is abelian-by-finite, which implies that each locally finite-by-abelian group is locally abelian-by-finite.

\begin{proposition}\label{p9.2}  If a locally abelian-by-finite group $G$ has finite free rank $r_0(G)$, then each subgroup $A\subset G$ of free rank $r_0(A)=r_0(G)$ has locally finite index in $G$.
\end{proposition}

\begin{proof} Let $A\subset G$ be a subgroup of free rank $r_0(A)=r_0(G)<\infty$. Then $A$ contains a free abelian subgroup of rank $r_0(A)$ and we lose no generality assuming that the subgroup $A$ is free abelian. We need to show that $A$ has finite index in each finitely generated subgroup $H$ that contains $A$. Since $G$ is locally abelian-by-finite, the finitely generated subgroup $H$ is abelian-by-finite and hence contains a finitely generated abelian subgroup $E$ of finite index in $H$. Then the subgroup $A\cap E$ has finite index in $A$ and thus $r_0(A)=r_0(A\cap E)\le r_0(E)\le r_0(G)=r_0(A)$. It follows from $r_0(A\cap E)=r_0(E)$ that the subgroup $A\cap E$ has finite index in $E$ and hence has finite index in $H$.
\end{proof}

The preceding proposition will be used in the proof of the following characterization of the free rank.

\begin{proposition}\label{p9.3} The free rank $r_0(G)$ of a countable locally finite-by-abelian group $G$ coincides with the smallest cardinality of a subset $S\subset G$ generating a subgroup $\langle S\rangle$ of locally finite index in $G$.
\end{proposition}

\begin{proof} Let $S\subset G$ be a subset of smallest possible cardinality $|S|$ generating a subgroup $\langle S\rangle$ of locally finite index in $G$. Proposition~\ref{p9.2} implies that  $|S|\le r_0(G)$.

It remains to prove that $r_0(G)\le|S|$. Assume conversely that $r_0(G)>|S|$ and find a subgroup $H\subset G$ isomorphic to $\IZ^{|S|+1}$. Let $F$ be the subgroup generated by $H\cup S$. Since $G$ is locally finite-by-abelian, the finitely generated subgroup $F$ of $G$ is finite-by-abelian. Consequently, it contains a finite normal subgroup $N\subset F$ with abelian quotient $F/N$. Replacing $N$ by a larger finite normal subgroup, we can assume the quotient $F/N$ is free abelian. Let $q:F\to F/N$ be the quotient homomorphism. It follows that $q(\langle S\rangle)$ is a free abelian subgroup of free rank $r_0(q(\langle S\rangle))\le |S|<r_0(H)=r_0(q(H))$. This implies that $q(\langle S\rangle)$ has infinite index in $F/N$ and hence $\langle S\rangle$ has infinite index in $F$, which contradicts the choice of $S$. This contradiction shows that $r_0(G)\le |S|$.
\end{proof}

In fact, the free rank $r_0(G)$ of a locally abelian-by-finite group $G$ coincides with the asymptotic dimension $\asdim(G)$ of $G$. Let us recall that a metric space $X$ has {\em asymptotic dimension} $\asdim(X)\le n$ for some $n\in\w\cup\{\infty\}$ if for each $\e<\infty$ there is a cover $\U$ of $X$ such that $\sup_{U\in\U}\diam(U)<\infty$ and each $\e$-ball $O_\e(x)$, $x\in X$, meets at most $n+1$ elements of the cover $\U$. The {\em asymptotic dimension} of a metric space $X$ is equal to the smallest number $n\in\w\cup\{\infty\}$ such that $\asdim(X)\le n$. It is known \cite[p.129]{Roe} that coarsely equivalent metric spaces have the same asymptotic dimension. As expected, the free Abelian group $\IZ^n$ has asymptotic dimension $\asdim(\IZ^n)=n$. This fact implies:

\begin{proposition}\label{p9.4} The asymptotic dimension $\asdim(G)$ of a countable locally abelian-by-finite group $G$ coincides with the free rank $r_0(G)$ of $G$.
\end{proposition}

\begin{proof} It follows from $\asdim(\IZ^n)=n=r_0(\IZ^n)$ that $r_0(G)\le \asdim(G)$. To prove that the reverse inequality, assume that $r_0(G)<\asdim(G)$. By Theorem 2.1 of \cite{DS}, the asymptotic dimension $\asdim(G)$ is equal to the least upper bound of asymptotic dimensions of finitely generated subgroups of $G$. Consequently, we can find a finitely generated subgroup $H\subset G$ such that $\asdim(H)>r_0(G)$. Since $G$ is locally abelian-by-finite, the finitely generated subgroup $H$ is abelian-by-finite and hence it contains a (finitely generated) abelian subgroup $A\subset H$ of finite index in $H$. Replacing $A$ by a suitable subgroup of finite index, we can assume that the group $A$ is free abelian. Since $A$ has finite index in $H$, the identity embedding $A\to H$ is a coarse equivalence and hence $r_0(A)=\asdim(A)=\asdim(H)>r_0(G)\ge r_0(A)$, which is a desired contradiction completing the proof.
\end{proof}

Next, we detect shift-homogeneous spaces among countable groups endowed with a proper left-invariant metric.

\begin{proposition}\label{p9.5} Each countable locally finite-by-abelian group $G$ endowed with a proper left-invariant metric is shift-homogeneous.
\end{proposition}

\begin{proof} Endow the group $G$ with a proper left-invariant metric $d$. To show that the metric space $(G,d)$ is shift-homogeneous, take any two elements $a,b\in G$ and consider the left shift $f:G\to G$, $f:x\mapsto ba^{-1}x$. It is clear that $f(a)=b$. The left-invariance  of the metric $d$ implies that $f$ is an isometry of $G$. To show that $f$ is a shift-isometry, we should check that $d(f|L,\id)<\infty$ for each large scale connected subset $L\subset G$. It follows that the set $L\cup\{ba^{-1}\}$ is contained in a finitely-generated subgroup $H\subset G$. Since the group $G$ is locally finite-by-abelian, its finitely-generated subgroup $H$ is finite-by-abelian and hence it contains a finite normal subgroup $N\subset H$ with abelian quotient $H/N$. Then the set $F=(ba^{-1})^H=\{hba^{-1}h^{-1}:h\in H\}\subset ba^{-1}N$ is finite. Observe that for each $x\in L\subset H$ we get $f(x)=ba^{-1}x=xx^{-1}ba^{-1}x\in xF$ and hence $d(f(x),x)\le\diam (F\cup\{1_G\})<\infty$. So, the left shift $f$ is a shift-isometry of $G$ and the metric space $(G,d)$ is shift-homogeneous.
\end{proof}

Finally let us prove a metrization lemma.

\begin{lemma}\label{l9.6} For each infinite finitely generated subgroup $H$ of locally finite index in a countable group $G$ there is a proper left-invariant metric $d$ on $G$ such that $H$ coincides with the factorizing component $G_{s(G)}$ of $G$.
\end{lemma}

\begin{proof} Let $F_0=\{1_G\}$ and $F_1=F_1^{-1}$ be a finite subset generating the subgroup $H$. Choose a sequence $(F_n)_{n=2}^\infty$ of finite sets such that $G=\bigcup_{n=2}^\infty F_n$ and $F_{n-1}\cdot F_{n-1}\subset F_n=F_n^{-1}$ for all $n\ge 2$. Then the proper left-invariant metric $$d(x,y)=\min\{n\in\w:y^{-1}x\in F_n\}$$on $G$ has the required property: $s(G)=1$ and the factorizing component $G_{s(G)}=G_{s(G)}(1_G)$ of $(G,d)$ coincides with the subgroup $H$.
\end{proof}

\section{Proof of Theorem~\ref{t1.2}}\label{s10}

Let $G_1,G_2$ be two countable locally finite-by-abelian groups. In each group $G_i$ fix a subset $S_i\subset G_i$ of the smallest possible cardinality generating a subgroup $H_i=\langle S_i\rangle$ of locally finite index in $G_i$. The minimality of $S_i$ implies that the subgroup $H_i$ either is infinite or trivial. We should prove that the groups $G_1$ and $G_2$ are coarsely isomorphic if and only if one of the following conditions holds:
\begin{enumerate}
\item $r_0(G_1)=r_0(G_2)=\infty$;
\item $r_0(G_1)=r_0(G_2)=0$ and $\fact_{G_1{/}H_1}=\fact_{G_2{/}H_2}$;
\item $0<r_0(G_1)=r_0(G_2)<\infty$ and $\fact_{G_1{/}H_1}=^*\fact_{G_2{/}H_2}$.
\end{enumerate}

To prove the ``only if'' part, assume that the groups $G_1,G_2$ are coarsely isomorphic. Then $r_0(G_1)=\asdim(G_1)=\asdim(G_2)=r_0(G_2)$ by Proposition~\ref{p9.4} and the invariance of the asymptotic dimension under coarse equivalences.

If $r_0(G_1)=r_0(G_2)=\infty$, then the condition (1) holds and we are done.

If $r_0(G_1)=r_0(G_2)=0$, then the groups $G_1$, $G_2$ are locally finite, the subgroups $H_1$ and $H_2$ are trivial, and
the equality $\fact_{G_1/H_1}=\fact_{G_1}=\fact_{G_2}=\fact_{G_2/H_2}$ follows from Theorem~\ref{t2.1} and Proposition~\ref{p2.3}.

It remains to consider the case $0<r_0(G_1)=r_0(G_2)<\infty$. In this case the groups $G_1,G_2$ are not locally finite and the subgroups $H_1$ and $H_2$ are infinite. Using Lemma~\ref{l9.6}, endow each group $G_i$ with a proper left-invariant metric $d_{G_i}$ such that the subgroup $H_i$ coincides with the factorizing component $(G_i)_{s(G_i)}$. In this case $\fact_{G_i/H_i}=\fact_{G_i}$.
It is well-known that finite-by-abelian groups are amenable, which implies that the groups $H_1,H_2$ (seen as homogeneous metric spaces) are amenable. By Proposition~\ref{p9.5} the groups $G_1,G_2$ are shift-homogeneous. So, it is legal to apply
 Theorem~\ref{t1.5} and conclude that
$\fact_{G_1/H_1}=\fact_{G_1}=^* \fact_{G_2}=\fact_{G_2/H_2}$.
\smallskip

To prove the ``if'' part of Theorem~\ref{t1.2}, assume that one of the conditions (1)--(3) holds.

If $r_0(G_1)=r_0(G_2)=\infty$, then the groups $G_1,G_2$ have infinite asymptotic dimension according to Proposition~\ref{p9.4}. By \cite[6.3]{BHZ},  each countable locally finite-by-abelian group is coarsely isomorphic to an abelian group and by \cite[6.4]{BHZ}, any two countable abelian groups of infinite asymptotic dimension are coarsely isomorphic, which implies that the groups $G_1,G_2$ are coarsely isomorphic.

If $r_0(G_1)=r_0(G_2)=0$ and $\fact_{G_1/H_1}=\fact_{G_2/H_2}$, then the groups $G_1,G_2$ are locally finite and the subgroups $H_1,H_2$ are trivial. Then the equality $\fact_{G_1}= \fact_{G_1/H_2}=\fact_{G_2/H_2}=\fact_{G_2}$ implies that $G_1$ and $G_2$ are coarsely isomorphic, see Theorem~\ref{t2.1} and Proposition~\ref{p2.3}.

Finally assume that $0<r_0(G_1)=r_0(G_2)<\infty$ and $\fact_{G_1/H_1}=^* \fact_{G_2/H_2}$. Using Lemma~\ref{l9.6}, endow each group $G_i$ with a proper left-invariant metric $d_i$ such that the subgroup $H_i$ coincides with the factorizing component $(G_i)_{s(G_i)}$. In this case the space $G_i/H_i=\{xH_i:x\in G_i\}$ can be identified with the ultrametric space $G_i/s(G_i)$. By Theorem~\ref{t4.1}, the group $G_i$ is coarsely isomorphic to $H_i\times (G_i/H_i)$ and by Corollary~\ref{c2.2} the ultrametric space $G_i/H_i$ is coarsely isomorphic to $\IZ_{\fact_{G_i/H_i}}$.
By (the proof of) Proposition~\ref{p9.3}, the subgroups $H_1,H_2$ have free rank $r_0(H_1)=r_0(G_1)=r_0(G_2)=r_0(H_2)$.
Being finitely generated, these subgroups are finite-by-abelian and hence abelian-by-finite. Consequently, each group $H_i$ contains a normal abelian subgroup $A_i$ with finite index. Replacing $A_i$ by a smaller subgroup of finite index, we can assume that $A_i$ is free abelian. By \cite[5.1]{BHZ}, the group $H_i$ is coarsely isomorphic to $A_i\times (H_i/A_i)$ and hence to $\IZ^{r_0(G_i)}\times H_i/A_i$.
By \cite[5.5]{BHZ}, the group $\IZ$ is coarsely isomorphic to the product $\IZ\times\IZ_k$ for any $k\in\IN$, which implies that the product $\IZ^{r_0(G_i)}\times H_i/A_i$ is coarsely isomorphic to $\IZ^{r_0(G_i)}$. Then the group $G_i$ is coarsely isomorphic to $\IZ^{r_0(G_i)}\times\IZ_{\fact_{G_i/H_i}}$.

Since $\fact_{G_1/H_1}=^* \fact_{G_2/H_2}$, there are natural numbers $n_1,n_2\in\IN$ such that $\fact_{n_1}+\fact_{G_1/H_1}=\fact_{n_2}+\fact_{G_2/H_2}$. Theorem~\ref{t2.1} implies that the groups $\IZ_{n_1}\times \IZ_{\fact_{G_1/H_2}}$ and $\IZ_{n_2}\times \IZ_{\fact_{G_2/H_2}}$ are coarsely isomorphic.  Since $r_0(G_i)>0$, the group $\IZ^{r_0(G_i)}$  is coarsely isomorphic to $\IZ_{n_i}\times\IZ^{r_0(G_i)}$.
So, we obtain the following chain of coarse isomorpisms:
$$
\begin{aligned}
G_1&\cong H_1\times(G_1/H_1)\cong H_1\times \IZ_{\fact_{G_1/H_1}}\cong \IZ^{r_0(G_1)}\times \IZ_{\fact_{G_1/H_1}}\cong\IZ^{r_0(G_1)}\times\IZ_{n_1}\times \IZ_{\fact_{G_1/H_1}}\cong\\
 &\cong\IZ^{r_0(G_2)}\times \IZ_{n_2}\times\IZ_{\fact_{G_2/H_2}}\cong \IZ^{r_0(G_2)}\times \IZ_{\fact_{G_2/H_2}}\cong H_2\times(G_2/H_2)\cong G_2.
\end{aligned}
$$

\section{Acknowledgements}

The authors would like to thank Jose Manuel Higes L\'opez for valuable discussions during the initial stages of work on this paper.

\end{document}